\documentclass[12pt,reqno]{article}

\usepackage[usenames]{color}
\usepackage{amssymb}
\usepackage{amsmath}
\usepackage{amsthm}
\usepackage{amsfonts}
\usepackage{amscd}
\usepackage{graphicx}

\usepackage[colorlinks=true,
linkcolor=webgreen,
filecolor=webbrown,
citecolor=webgreen]{hyperref}

\definecolor{webgreen}{rgb}{0,.5,0}
\definecolor{webbrown}{rgb}{.6,0,0}

\usepackage{color}
\usepackage{fullpage}
\usepackage{float}

\usepackage{graphics}
\usepackage{latexsym}
\usepackage{epsf}

\newcommand{\seqnum}[1]{\href{https://oeis.org/#1}{\rm \underline{#1}}}
\def\modd#1 #2{#1\ \mbox{\rm (mod}\ #2\mbox{\rm )}}

\begin{document}

\theoremstyle{plain}
\newtheorem{theorem}{Theorem}
\newtheorem{corollary}[theorem]{Corollary}
\newtheorem{lemma}[theorem]{Lemma}
\newtheorem{proposition}[theorem]{Proposition}

\theoremstyle{definition}
\newtheorem{definition}[theorem]{Definition}
\newtheorem{example}[theorem]{Example}
\newtheorem{conjecture}[theorem]{Conjecture}

\theoremstyle{remark}
\newtheorem{remark}[theorem]{Remark}

\title{Beatty Sequences for a Quadratic Irrational:  Decidability and Applications}

\def\Que{\mathbb{Q}}
\def\Enn{\mathbb{N}}
\def\Zee{\mathbb{Z}}
\def\AND{\ \wedge\ }
\def\suchthat{\, : \,}
\def\andd{\, \wedge \, }
\def\shift{{\tt shift}}

\author{Luke Schaeffer\\
Institute for Quantum Computing\\
University of Waterloo\\
Waterloo, ON N2L 3G1 \\
Canada\\
\href{mailto:lschaeff@uwaterloo.ca}{\tt lschaeff@uwaterloo.ca}
\and
Jeffrey Shallit and Stefan Zorcic\\
School of Computer Science\\
University of Waterloo\\ 
Waterloo, ON N2L 3G1\\
Canada\\
\href{mailto:shallit@uwaterloo.ca}{\tt shallit@uwaterloo.ca}\\
\href{mailto:szorcic@uwaterloo.ca}{\tt szorcic@uwaterloo.ca}
}

\maketitle

\begin{abstract}
Let $\alpha$ and $\beta$ be real numbers belonging to the same quadratic field.
We show that the inhomogeneous
Beatty sequence $(\lfloor n \alpha + \beta \rfloor)_{n \geq 1}$
is synchronized, in the sense that there is a finite automaton that takes as input
the Ostrowski representations of $n$ and $y$ in parallel, and
accepts if and only if $y =  \lfloor n \alpha + \beta \rfloor$.  Since it is already known that the addition relation is computable for Ostrowski representations based on a quadratic irrational,
a consequence is a new and rather simple proof that the first-order logical theory of these
sequences with addition is decidable.   The decision procedure is easily implemented in the free software
{\tt Walnut}.

As an application, we show
that for each $r \geq 1$ it is decidable whether the set 
$\{  \lfloor n \alpha + \beta \rfloor \suchthat n \geq 1 \}$
forms an additive basis (or asymptotic additive basis) of order $r$.
Using our techniques, we also solve some open problems of
Reble and Kimberling, and give an explicit characterization of a sequence of Hildebrand
et al.
\end{abstract}

\section{Introduction}

Let $\alpha$ be an irrational number.   The sequence
$(\lfloor n \alpha \rfloor)_{n \geq 1}$ is
known as a {\it Beatty sequence\/}
\cite{Beatty:1926} and has been very widely studied.\footnote{Some
authors \cite{Dekking:2023} insist that $\alpha$ must be greater than $1$, but
here we do not. Also, sometimes it is useful to allow the index $n$ to be $0$.}   The two most famous Beatty sequences correspond to
$\alpha = (1+\sqrt{5})/2$ and $\alpha= (3+\sqrt{5})/2$; they are 
also known as the
{\it lower Wythoff\/} and {\it upper Wythoff\/} sequences, respectively, and they
form sequences
\seqnum{A000201} and \seqnum{A001950} in the {\it On-Line Encyclopedia of Integer Sequences} (OEIS)
\cite{Sloane:2026}.

In an {\it inhomogeneous Beatty sequence}\footnote{Sometimes called {\it non-homogeneous Beatty sequence\/} in the literature; see \cite{Fraenkel:1969}.}, we permit a constant
additive factor
inside the floor brackets:  $(\lfloor n \alpha + \beta \rfloor)_{n \geq 1}$.
The first difference of such a sequence forms a
{\it Sturmian sequence}, which also has been widely
studied with a huge literature \cite{Lothaire:2002,Berstel&Lauve&Reutenauer&Saliola:2009}.

A different generalization of Beatty sequences was proposed by
Allouche and Dekking \cite{Allouche&Dekking:2019}.
This concerns sequences of the form
$(p \lfloor n \alpha \rfloor + qn + r)_{n \geq 0}$, where $p, q, r$ are
natural numbers.

In this paper, we show that if $\alpha$ and $\beta$ belong to the same quadratic field, then the inhomogeneous
Beatty sequence $(\lfloor n \alpha + \beta \rfloor)_{n \geq 1}$
is synchronized, in the sense that there is a finite automaton that takes
the Ostrowski representations of $n$ and $y$ in parallel, and
accepts if and only if $y =  \lfloor n \alpha + \beta \rfloor$.  Furthermore,
the converse of this theorem also holds.
Ostrowski representations are discussed in Section~\ref{ostrowski}; for an introduction and explanation of how automata can use them, see \cite{Baranwal&Schaeffer&Shallit:2021}. Synchronized sequences are discussed in Section~\ref{synch}; for a summary, see \cite{Shallit:2021h}.  As a consequence, we get a new and rather simple proof that first-order logical theory of quadratic Beatty sequences with addition is decidable.  History of this result is discussed in Section~\ref{history}, and our principal results
appear in Section~\ref{main}: Theorem~\ref{main1} and Corollary~\ref{main1a}.   Not only are these first-order statements decidable, but in practice many theorems about individual Beatty sequences can be proved in a matter of seconds with an implementation of the decision procedure in the free software {\tt Walnut} \cite{Mousavi:2016,Shallit:2023}.   {\tt Walnut} is available at \\
\centerline{\url{https://cs.uwaterloo.ca/~shallit/walnut.html} .}

In Section~\ref{additive}, we apply these results to problems of additive number theory.   We show
that for each $r \geq 1$ it is decidable whether the set
$\{  \lfloor n \alpha + \beta \rfloor \suchthat n \geq 1 \}$
forms an additive basis (or asymptotic additive basis) of order $r$.  

In Section~\ref{reble}
we use these ideas to solve some open problems of
Don Reble, and in Section~\ref{graham} we show a question of Graham is
decidable for quadratic irrationals.
In Section~\ref{hildebrand} we give an explicit
description of a sequence studied by Hildebrand, Li, Li, and Xie.   In Sections~\ref{kimberling} and \ref{sec11}, we solve some open problems of Kimberling.   
In Section~\ref{sec12}, 
we prove a number of identities about sequences in the OEIS.
In Section~\ref{fractional} we discuss fractional parts. Finally,
in Section~\ref{converse-sec} we
prove the converse (Theorem~\ref{main1-converse}) of Theorem~\ref{main1}.

\section{Previous work}
\label{history}

In this section, we briefly survey previous work.
Berstel \cite{Berstel:1982,Berstel:1986b} was apparently the first to show that a finite automaton can recognize the addition relation
$x+y=z$ when $x, y,$ and $z$ are given in Fibonacci
representation (see Section~\ref{ostrowski}).  This
was generalized to the case of certain kinds of linear
recurrences by Frougny \cite{Frougny:1988,Frougny:1991b,Frougny:1992a}.   Mousavi et al.\ \cite{Mousavi&Schaeffer&Shallit:2016} showed how to implement Fibonacci representations in the {\tt Walnut} theorem-proving system and, with it, proved many results in combinatorics on words ``purely mechanically" \cite{Du&Mousavi&Rowland&Schaeffer&Shallit:2017,Du&Mousavi&Schaeffer&Shallit:2016}.

A related representation, called Pell representation, was studied by Baranwal and Shallit \cite{Baranwal&Shallit:2019}; it is also implemented in {\tt Walnut}.

More generally, Hieronymi and Terry \cite{Hieronymi&Terry:2018} and, independently
Baranwal \cite{Baranwal:2020} and
Baranwal et al.\ \cite{Baranwal&Schaeffer&Shallit:2021} showed that a finite automaton can do addition in the Ostrowski numeration system (see Section~\ref{ostrowski}) for quadratic numbers.  This was implemented in {\tt Walnut} and also
{\tt Pecan}, another software system \cite{Hieronymi:2022}.

As a consequence, the first-order theory of $\Enn = \{0,1,2,\ldots\}$ together with the Sturmian
sequence ${\bf f} = 01001010 \cdots$ (the Fibonacci word, corresponding to $(3-\sqrt{5})/2$)
with addition is decidable.  This was observed by
Bruy\`ere and Hansel \cite{Bruyere&Hansel:1997}, building on work of Bruy\`ere et al.~\cite{Bruyere&Hansel&Michaux&Villemaire:1994}.  Later, these ideas were used by Mousavi et al.\ \cite{Mousavi&Schaeffer&Shallit:2016} to solve questions about
the properties of $\bf f$ and related words.  Also see \cite{Hieronymi:2016}.

More generally, the first-order theory of $\Enn$ together with the Sturmian sequence for any quadratic number is decidable.  This was first observed by Hieronymi and Terry \cite{Hieronymi&Terry:2018}, and independently by
Baranwal \cite{Baranwal:2020} and
Baranwal et al.\ \cite{Baranwal&Schaeffer&Shallit:2021}.  

Similarly, the first-order theory of $\Enn$ with addition  together with the
Beatty sequence for the golden ratio is decidable \cite{Mousavi&Schaeffer&Shallit:2016,Hieronymi:2016,Shallit:2023}.   Recently, Khani and Zarei \cite{Khani&Zarei:2023} proved decidability of this theory using a different technique, quantifier elimination.

More generally, the analogous decidability result is true for the Beatty sequence of any quadratic number.  This is implied
by the results of  Hieronymi et al.~\cite{Hieronymi:2022} but not explicitly stated there.  For related results, see \cite{Khani&Valizadeh&Zarei:2022,Gunaydin&Ozsahakyan:2022}.

The novelty of our results is the simplicity of the method (see, for example, Theorem~\ref{thm4} and Corollary~\ref{five}) and our applications to solving problems in the theory of Beatty sequences.

\section{Ostrowski \texorpdfstring{$\alpha$}{α}-representation}
\label{ostrowski}

In the famous Fibonacci (or Zeckendorf) system for representing natural numbers \cite{Lekkerkerker:1952,Zeckendorf:1972}, every non-negative integer is written uniquely as a sum of distinct Fibonacci numbers, subject to the condition that we never use two consecutive Fibonacci numbers.   More precisely, if the Fibonacci numbers are (as usual) $F_0 = 0$, $F_1 = 1$, and $F_n = F_{n-1} + F_{n-2}$ for $n \geq 2$, then a natural number $n$ can be expressed as a binary string $e_t \cdots e_1$ with interpretation
$n = \sum_{1 \leq j \leq t} e_j F_{j+1}$.  Notice that the 
most significant digit is written at the left.
\vskip .03in

The Ostrowski numeration system is a generalization of this idea.
Let $ \alpha > 0$ be an irrational real number, with continued
fraction $[a_0, a_1, a_2, \ldots ]$.   Define, as usual, the
$i$'th convergent $p_i/q_i$ to be $[a_0, a_1, \ldots, a_i]$.
Recall (e.g., \cite{Hardy&Wright:1985}) that $p_i = a_i p_{i-1} + p_{i-2}$ and
$q_i = a_i q_{i-1} + q_{i-2}$ with initial conditions
$p_{-2} = q_{-1} = 0$ and $q_{-2} = p_{-1} = 1$.

The {\it Ostrowski $\alpha$-numeration system\/} is a method for
uniquely expressing every natural number $n$ as an integer linear
combination of the $q_i$, namely
$\sum_{0 \leq i \leq t} e_i q_i$. Here, to obtain uniqueness of the representation,  the $e_i$ are restricted by the following
inequalities:
\begin{itemize}
\item[(a)] $0 \leq e_i \leq a_{i+1}$ for $i \geq 1$;
\item[(b)] $e_i = a_{i+1}$ implies $e_{i-1} = 0$ for $i \geq 1$;
\item[(c)] $0 \leq e_0 < a_1$.
\end{itemize}
See \cite{Ostrowski:1922} and \cite[\S 3.9]{Allouche&Shallit:2003}.
We write the canonical Ostrowski $\alpha$-representation of $n$, starting with
the most significant digit, and satisfying the rules (a)--(c) above, as
$(n)_\alpha$.   The inverse mapping, which sends a string
$x = e_t e_{t-1} \cdots e_0$ to the sum $\sum_{0 \leq i \leq t} e_i q_i$,
is denoted by $[x]_\alpha$.  The Fibonacci numeration system then is
a special case of Ostrowski $\alpha$-numeration, where 
$\alpha = (3-\sqrt{5})/2$.

In some cases we need to write the $\alpha$-representation of pairs or
tuples.  We do so by working over a product alphabet, and padding the
shorter representation at the front with $0$'s, if needed.

\section{Automata, synchronized sequences, and first-order logic}
\label{synch}

We say a sequence $(s_n)_{n \geq 0}$ is {\it automatic\/} if there
is a deterministic finite automaton with output (DFAO) that takes
as input the representation of $n$ in some addable numeration system, and
$s_n$ is the output associated with the last state reached
\cite{Allouche&Shallit:2003}.  Here by `addable', we mean that the
addition relation $x+y=z$ is recognizable by an automaton
for this numeration system; this is the case for base-$k$ where $k\geq 2$
is an integer, the Fibonacci and Tribonacci numeration systems, and many
others.  In particular, this is the case for Ostrowski $\alpha$-representation
when $\alpha$ is a quadratic irrational \cite{Baranwal:2020,Baranwal&Schaeffer&Shallit:2021}.
It is clear that such a sequence takes only finitely many values.

We say that a sequence $(s_n)_{n \geq 0}$ is {\it synchronized\/}
if there is a finite automaton that takes the representations of
$n$ and $y$ in parallel, and
accepts if and only if $y = s_n$.  In this paper, the particular representations
we use are the Ostrowski $\alpha$-representations discussed in the
previous section.  Moreover, in this paper, we always assume that $n$ and $y$
are represented in the same numeration system.

Synchronized sequences have a number of nice closure properties.   These can be found in \cite{Shallit:2021h} for $k$-synchronized sequences, but the same proofs also suffice for sequences defined in terms of Ostrowski representations.
\begin{proposition}
\leavevmode
\begin{itemize}
\item[(a)]  The sequence $(\lfloor (bn+c)/d \rfloor)_{n \geq 1}$ is synchronized for integers $b,c,d$ with $b \geq \max(1, -c)$ and 
$d \geq 1$.
\item[(b)] If $(f(n))_{n \geq 1}$ and $(g(n))_{n \geq 1}$
are both synchronized, then so is their termwise sum $(f(n)+g(n))_{n \geq 1}$.
\item[(c)] If $(f(n))_{n \geq 1}$ and $(g(n))_{n \geq 1}$
are both synchronized then so is
their composition $(f(g(n)))_{n \geq 1}$.
\end{itemize}
\label{prop1}
\end{proposition}

From the results of \cite{Bruyere&Hansel&Michaux&Villemaire:1994} we immediately get the following result.   
\begin{theorem}
There is an algorithm that takes as input
a first-order logical formula on natural
numbers for $f(n)$,
defined in terms of addition, subtraction, multiplication by constants,
automatic sequences, synchronized sequences, and computes a synchronized
automaton for $f$.
\label{decide}
\end{theorem}
Furthermore, the {\tt Walnut} theorem-prover can compute these synchronized automata.
For more information about {\tt Walnut}, see \cite{Mousavi:2016,Shallit:2023}.
Thus, to prove our main result, we only need to show that
the map $n \rightarrow \lfloor n \alpha + \beta \rfloor$
is synchronized.

As a consequence of Theorem~\ref{decide}, we get the following corollary.
Recall that the subword complexity (aka factor complexity) of a sequence
is the function counting the number of distinct contiguous blocks
appearing in it.
\begin{corollary}
Let ${\bf f} = (f(n))_{n \geq 0}$ be an automatic sequence.
Then its subword complexity function $\rho_{\bf f} (n)$ 
satisfies $\rho_{\bf f} (n) = O(n)$.
\label{four}
\end{corollary}

\begin{proof}
As explained in 
\cite[\S 10.5]{Shallit:2023}, given an automatic sequence ${\bf f} = (f(n))_{n \geq 0}$
there is a first-order logical formula for the
appearance function $A_{\bf f}(n)$, which is the length of the
shortest prefix of $\bf f$ containing an occurrence of each of its length-$n$
factors.  Hence, by Theorem~\ref{decide}, it follows that there is a
synchronized automaton for $A_{\bf f}(n)$.  Then from 
\cite[Thm.~8]{Shallit:2021h}, it follows that $A_{\bf f}(n) = O(n)$.  (The
proof there technically only was stated for $k$-synchronized sequences,
but it goes through in exactly the same way for the more general
notion of synchronized sequence over an addable numeration
system.)  Hence $\rho_{\bf f}(n)$ is also $O(n)$.
\end{proof}

\section{The main result and its implementation}
\label{main}

We need a simple lemma.
\begin{lemma}
Let $x$ be a real number and $c \geq 1$ an integer.
Then $\lfloor x/c \rfloor = \lfloor {{\lfloor x \rfloor} \over c} \rfloor$.
\label{simple}
\end{lemma}

\begin{proof}
Using Euclidean division, write $x = qc + r$ for $q \in \Zee$ and
$0 \leq r < c$.   Since $0 \leq \lfloor r \rfloor < c$, we have
\begin{equation}
{{\lfloor x \rfloor}\over c} = {{qc + \lfloor r \rfloor} \over c} \\
= q + {{\lfloor r \rfloor } \over c} .
\end{equation}
Hence 
$\lfloor {{\lfloor x \rfloor} \over c} \rfloor = q$.
On the other hand $x/c = q + r/c$, so
$\lfloor x/c \rfloor = q$.
\end{proof}

We say that a continued fraction is {\it purely periodic\/} of
{\it period length\/} $P$ if it is of the form
\begin{equation}
[a_1, a_2, \ldots, a_P, a_1, a_2, \ldots, a_P, a_1, a_2, \ldots, a_P,
\ldots]
\label{cfpp}
\end{equation}
for some $P \geq 1$.  

\begin{theorem}
Suppose $0 < \gamma < 1$ is a quadratic irrational
such that $1/\gamma$ has a purely periodic continued
fraction of the form given in Eq.~\eqref{cfpp}. 
Let $p_i/q_i$ denote the $i$'th convergent to the
continued fraction for $\gamma$.
Then for all integers $n\geq 1$ we have
$$[ (n-1)_\gamma\, 0^P ]_\gamma = q_{P} (n-1) + q_{P-1} \cdot
\lfloor n \gamma \rfloor  .$$
\label{thm4}
\end{theorem}

\begin{proof}
We use the $2\times 2$ matrix approach to continued fractions,
pioneered by Hurwitz \cite{Hurwitz&Kritikos:1986}, Frame \cite{Frame:1949}, and Kolden \cite{Kolden:1949}, and also discussed (for example)
in \cite{Allouche&Shallit:2003}.

This approach uses the fact, easily proved by induction, that
\begin{equation}
\left[ \begin{array}{cc} a_0 & 1 \\ 1 & 0  \end{array} \right]
\left[ \begin{array}{cc} a_1 & 1 \\ 1 & 0  \end{array} \right]
\cdots
\left[ \begin{array}{cc} a_i & 1 \\ 1 & 0  \end{array} \right] =
\left[ \begin{array}{cc} p_i & p_{i-1} \\ q_i & q_{i-1}  \end{array}\right] 
\label{hur}
\end{equation}
for all $i \geq 0$.

Since $a_i = a_{i+P}$ for all $i \geq 1$, we have
\begin{align*}
\left[ \begin{array}{cc} p_{i+P} & p_{i+P-1} \\ q_{i+P} & q_{i+P-1}  \end{array}\right] &=
\left[ \begin{array}{cc} a_0 & 1 \\ 1 & 0  \end{array} \right]
\left[ \begin{array}{cc} a_1 & 1 \\ 1 & 0  \end{array} \right]
\cdots
\left[ \begin{array}{cc} a_{i+P} & 1 \\ 1 & 0  \end{array} \right] \\
&=
\left[ \begin{array}{cc} a_0 & 1 \\ 1 & 0  \end{array} \right]
\left[ \begin{array}{cc} a_1 & 1 \\ 1 & 0  \end{array} \right]
\cdots
\left[ \begin{array}{cc} a_{P} & 1 \\ 1 & 0  \end{array} \right]  \quad
\left[ \begin{array}{cc} a_{P+1} & 1 \\ 1 & 0  \end{array} \right]
\cdots
\left[ \begin{array}{cc} a_{i+P} & 1 \\ 1 & 0  \end{array} \right] \\
&= \left[ \begin{array}{cc} a_0 & 1 \\ 1 & 0  \end{array} \right]
\left[ \begin{array}{cc} a_1 & 1 \\ 1 & 0  \end{array} \right]
\cdots
\left[ \begin{array}{cc} a_{P} & 1 \\ 1 & 0  \end{array} \right]  \quad
\left[ \begin{array}{cc} a_1 & 1 \\ 1 & 0  \end{array} \right]
\cdots
\left[ \begin{array}{cc} a_i & 1 \\ 1 & 0  \end{array} \right] \\
&= \left[ \begin{array}{cc} p_P & p_{P-1} \\ q_P & q_{P-1}  \end{array}\right] 
\left[ \begin{array}{cc} q_i & q_{i-1} \\ p_i & p_{i-1} \end{array}\right] ,
\end{align*}
where we have used the fact that $a_0 = 0$.

Let us now compare the entry in row $2$ and column $1$ of both sides.
We get
\begin{equation}
 q_{i+P} = q_P q_i + q_{P-1} p_i.
 \label{cf}
\end{equation}

Now suppose the Ostrowski $\gamma$-representation of
$n-1$ is $\sum_{0 \leq i \leq t} e_i q_i$.
Then from Eq.~\eqref{cf}, we get
\begin{align*}
[(n-1)_\gamma 0^P]_\gamma &= \sum_{0 \leq i \leq t} e_i q_{i+P} \\
&= \sum_{0 \leq i \leq t} e_i (q_P q_i + q_{P-1} p_i) \\
&= \left(q_P \sum_{0 \leq i \leq t} e_i q_i \right) + 
\left(q_{P-1} \sum_{0 \leq i \leq t} e_i p_i \right) \\
&=  q_{P} (n-1) + q_{P-1} \lfloor n \gamma \rfloor ,
\end{align*}
as desired.
Here, in the last step, we have used \cite[Corollary 9.1.14]{Allouche&Shallit:2003}.
\end{proof}

\begin{corollary}
Suppose $0 < \gamma < 1$ is a quadratic irrational
such that $1/\gamma$ has a purely periodic continued
fraction of the form given in Eq.~\eqref{cfpp}. 
Let $p_i/q_i$ denote the $i$'th convergent to the 
continued fraction for $\gamma$.
Then the sequence $(\lfloor n \gamma \rfloor)_{n \geq 1}$
is $\gamma$-Ostrowski synchronized.
\label{five}
\end{corollary}

\begin{proof}
It is easy to see that for each
$m \geq 1$ there is a finite automaton
${\tt shift}_P$ accepting $0^P y$ and $y 0^P$ in parallel. 
We sketch the details.   The states are $P$-tuples of
integers appearing in the Ostrowski representation of 
$\Enn$, and the initial state and unique final state is $0^P$.
Here the ``meaning'' of a state named $f_1\cdots f_P$ is
that these are the last $P$ symbols of $y$ that the automaton has read.
The transition function $\delta$ is
defined by $\delta(f_1 f_2 \cdots f_P, [f_1, g]) = f_2 \cdots f_P g$.

Since $q_P$ and $q_{P-1}$ are constants, the 
first-order formula
$$\exists u,v \ n=u+1 \andd \shift_P(u,v) \andd v=q_{P-1} \cdot z + q_{P}
\cdot u$$
expresses $z = \lfloor n \gamma \rfloor$ as a synchronized function of $n$, for $n \geq 1$.  The result now follows by Theorem~\ref{decide}.
\label{main2}
\end{proof}

\begin{remark} 
For the special case of $\alpha = (1+\sqrt{5})/2$, the golden ratio, this was
already proved in \cite{Mousavi&Schaeffer&Shallit:2016}.   
\end{remark}

\begin{remark}
In practice, the automaton for $\shift_P(u,v)$ can sometimes be more
easily constructed as follows.  We construct a single automaton
for $\shift_1(u,v)$, and then create the shifters for $P \geq 2$
by iterating the construction, using a sequence of first-order formulas:
$$\shift_{P+1}(u,v) := \exists x \ \shift_1(u,x) \AND \shift_P(x,v).$$
\end{remark}

We can now prove the two main results of the paper.

\begin{theorem}
Suppose $0 < \gamma < 1$ is a quadratic irrational
such that 
$$\gamma = [0, a_1, a_2, \ldots, a_P, a_1, a_2, \ldots, a_P, a_1, a_2, \ldots, a_P,
\ldots],$$
and suppose $\alpha, \beta \in \Que(\gamma)$
such that $\alpha \geq 0$ and $\alpha+\beta \geq 0$.
Then $(\lfloor n \alpha + \beta \rfloor)_{n \geq 1}$
is $\gamma$-Ostrowski synchronized.
\label{main1}
\end{theorem}

\begin{proof}
If $\alpha, \beta \in \Que(\gamma)$ satisfy the stated
inequalities, then we can write
$\alpha= (a+ b\gamma)/c$,
$\beta = (d + e\gamma)/c$
for some integers $a, b, c, d, e$ with $b\geq 0$ and $c \geq 1$.
So 
\begin{align*}
\lfloor \alpha n + \beta\rfloor &= \lfloor n (a+b\gamma)/c + (d + e\gamma)/c \rfloor  \\
&= \left\lfloor {{\lfloor bn\gamma  + an + e\gamma + d \rfloor} \over c } \right\rfloor \\
&= \left\lfloor {{ \lfloor (bn+e)\gamma \rfloor  + an+d } \over c} \right\rfloor.
\end{align*}
Now from Corollary~\ref{five}, we know that
$f(n) = \lfloor n \gamma \rfloor$ is $\gamma$-Ostrowski
synchronized.   Hence, by Proposition~\ref{prop1}, so is $f(bn+e)$.   Hence so is
$f(bn+e) + an+d$.   Hence so is
$\lfloor {{f(bn+e)+ an+d}\over c} \rfloor$.   The result is proved.
\end{proof}

\begin{remark} In Section~\ref{fractional} we prove the
converse to Theorem~\ref{main1}.
\end{remark}

\begin{corollary}
Let $0 < \gamma < 1$ be a quadratic irrational such that $1/\gamma$ has a purely
periodic continued fraction, and let $\alpha, \beta \in \Que(\gamma)$.    The first-order theory of the natural numbers with addition and the
function $n \rightarrow \lfloor n\alpha+ \beta \rfloor$ is decidable.
\label{main1a}
\end{corollary}

\begin{proof}
We combine the result of Theorem~\ref{main1} with Theorem~\ref{decide}.
\end{proof}

\section{Beatty sequences in {\tt Walnut}}

As mentioned previously, {\tt Walnut} is free software that implements a
decision procedure for automatic sequences.  It can work with both
DFAO's and ordinary finite automata.  Most queries in {\tt Walnut} are 
simply first-order logical statements, with some minor concessions
to allowing these statements to be formulated in plain ASCII.  For example,
the universal quantifier $\forall$ is written {\tt A}, and the
existential quantifier $\exists$ is written {\tt E}.  Also,
logical AND is written ${\tt \&}$, logical OR is written {\tt |},
negation is written ${\tt\char'176}$, implication is written {\tt =>},
and logical IFF is written {\tt <=>}.

The {\tt ost} command creates an Ostrowski numeration system, the {\tt def} command defines an automaton for later use, the {\tt reg} command allows defining a regular expression, and the {\tt eval} command evaluates
a formula with no free variables and reports {\tt TRUE} or
{\tt FALSE}.   In {\tt Walnut}, we can multiply and do integer
division by natural number constants, but not variables.  When an automaton
is defined using {\tt def}, the user supplies a name for it.  The automaton
can then be used later by invoking its name prefixed with a dollar sign.

If a first-order statement has no free variables, then {\tt Walnut} returns
either {\tt TRUE} or {\tt FALSE}, and its computations provide a proof of
the result.  Otherwise it returns an automaton accepting the values of the
free variables that make the logical statement evaluate to {\tt TRUE}.

The {\tt combine} command allows the user to combine different automata
into a single DFAO with specified outputs.

We now explain how to use Beatty sequences of quadratic irrationals in 
{\tt Walnut}.  Suppose $\alpha$ and $\beta$ lie in a quadratic field $\Que(\gamma)$, where $1/\gamma$ 
has a purely periodic continued fraction.

If the period of $\gamma$'s continued fraction starts with
$1$, rotate the period to get $\gamma'$ with a period that begins with a number larger than $1$.   Then $\gamma'$ generates the same
field as $\gamma$.  This is necessary because {\tt Walnut} assumes
that the continued fraction represents a number in the open
interval $(0, {1\over 2})$.
Rotation of the period only
fails when $\gamma = [0,1,1,1,\ldots]$, but in that
case we can just use the Fibonacci numeration system instead, based
on the real number $[0,2,1,1,1,\ldots]$.

Next, create the appropriate number system with a command
of the form 
\begin{verbatim}
ost <name> [0] [period]:
\end{verbatim}
This will create a number system called
{\tt msd\_name}, where you replace {\tt name} with any
name you like.  For example
\begin{verbatim}
ost sqrt7 [0] [4 1 1 1]:
\end{verbatim}

Use Theorem~\ref{thm4} to write a {\tt Walnut} formula
for a synchronized automaton accepting $n$ and
$\lfloor n \gamma \rfloor$ in parallel.
Finally, use the relationship between $\alpha, \beta, $
and $\gamma$, and Theorem~\ref{main1} to create a 
{\tt Walnut} formula for 
$\lfloor n  \alpha + \beta \rfloor$.

As an example, let us obtain the synchronized automaton
corresponding to the Beatty sequence
where $\alpha = {{\sqrt{21}-1} \over 2}$,
$\beta = {{\sqrt{21}+3} \over 4}$.
Since $\alpha = [1,1,3,1,3,\ldots]$, the periodic portion of the continued fraction 
is $1,3$, which we then rotate to get $[3,1]$.
Define $\gamma = [0,3,1,3,1,3,1,\ldots]$; then
$\gamma = (\sqrt{21}-3)/6$.   In the corresponding 
Ostrowski numeration system, the digits can be $0,1,2,$ or
$3$.   We now build the shifting automaton as described
in the proof of Corollary~\ref{five}, and it is depicted in Figure~\ref{shiftaut}.   We give it the name {\tt shift13}
in {\tt Walnut}.
\begin{figure}[htb]
\begin{center}
\includegraphics[width=6.5in]{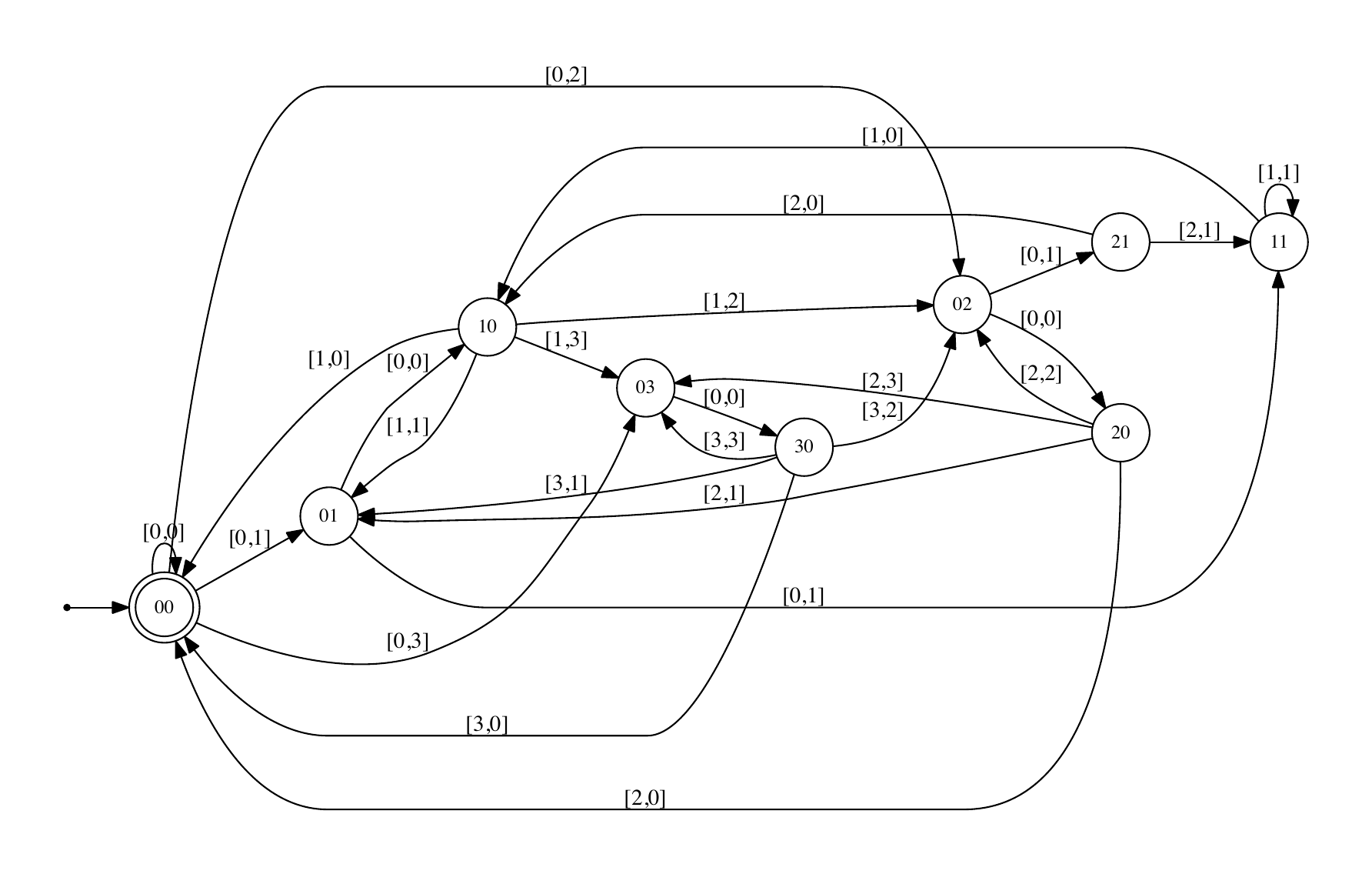}
\end{center}
\caption{Shift automaton.}
\label{shiftaut}
\end{figure}
States correspond to two consecutive Ostrowski digits that are permitted
by the inequalities given in Section~\ref{ostrowski}.

Next, we obtain the synchronized automaton for
$(\lfloor n \gamma \rfloor)_{n \geq 1}$.  In Theorem~\ref{thm4}, we have $q_P = 4$ and $q_{P-1} = 3$.
We use the following {\tt Walnut} commands:
\begin{verbatim}
ost s13 [0] [3 1]:
# define the numeration system
def beattyg "?msd_s13 (n=0 & z=0) | (Eu,v n=u+1 & $shift13(u,v) & v=3*z+4*u)":
\end{verbatim}
This automaton has 32 states.  Notice that we have to handle $n=0$ as a special case.

Now it is easy to see that $\alpha = 1+3 \gamma$
and $\beta = (3 + 3\gamma)/2$.   Thus in Theorem~\ref{main1}
we take $a = 2$, $b = 6$, $c = 2$, $d = 3$, and $e = 3$,
and enter the following {\tt Walnut} command:
\begin{verbatim}
def beatty "?msd_s13 Eu $beattyg(6*n+3,u) & z=(u+2*n+3)/2":
\end{verbatim}
This creates a synchronized automaton for 
$(\lfloor n \alpha+\beta  \rfloor)_{n \geq 1}$,
with 59 states.

We now claim that every integer $\geq 12$ is the sum
of two elements of $(\lfloor n \alpha+\beta  \rfloor)_{n \geq 1}$.   We can verify this as follows:
\begin{verbatim}
eval check2 "?msd_s13 An (n>=12) => Eu,v,n1,n2 u>=1 & 
   v>=1 & $beatty(u,n1) & $beatty(v,n2) & n=n1+n2":
\end{verbatim}
and {\tt Walnut} returns {\tt TRUE}.   Note the need to
specify that $u,v \geq 1$.

There are obviously some limits to our approach.  For example, as the
length of the repeating portion of the continued fraction gets longer,
the size of the automata grow rapidly, so that working with
periods of length (say) longer than $10$ might require computational
resources beyond one's computer.  

Another limitation involves theorems about Beatty sequences concerning
two numbers that are not in the same quadratic field.  There is no
obvious way to do this using \texttt{Walnut}.

\section{The additive theory of Beatty sequences}
\label{additive}

Given a sequence $(s_i)_{n \geq 1}$ or set $S = \{s_1, s_2,
\ldots \}$ of natural numbers,
a fundamental question of additive number theory is to decide 
which integers have representations as $r$-fold sums of the $s_i$
\cite{Nathanson:1996,Nathanson:2000}.  
More precisely, we define the set of $r$-fold
sums $A_s(r)$ as follows:
$$A_s(r) = \{ n \in \Enn \suchthat \exists i_1, i_2, \ldots, i_r
 \text{ such that } n = s_{i_1} + s_{i_2} + \cdots + s_{i_r} \}.$$
We say that $(s_i)_{n \geq 1}$ forms an {\it additive basis of
order $r$} if $A_s(r) = \Enn$.  We say that $(s_i)_{n \geq 1}$ forms
an {\it asymptotic additive basis of order $r$}
if $\Enn \setminus A_s(r)$ is finite.

In this section we generalize some recent results on the additive theory of Beatty sequences \cite{Kawsumarng:2021,Phunphayap&Pongsriiam&Shallit:2022,Shallit:2022,Dekking:2022}.

\begin{theorem}
Let $\gamma$ be a real quadratic irrational, and suppose
$\alpha, \beta \in \Que(\gamma)$.  Then
the inhomogeneous Beatty sequence $(\lfloor \alpha n + \beta\rfloor)_{n \geq 0}$
forms an asymptotic additive basis of order $h$ for some
integer $h \geq 1$.  Furthermore, the smallest such $h$ is
effectively computable, and one can also decide if
$(\lfloor \alpha n + \beta\rfloor)_{n \geq 0}$ forms an additive basis.
\end{theorem}

\begin{proof}
The first claim follows immediately from a classical theorem 
(see Nathanson \cite[p.~366, Theorems 11.6 and 11.7]{Nathanson:2000}).  
Furthermore, the proof of this classical result given in the cited
reference shows that every Beatty sequence
forms an asymptotic additive basis where an {\it upper bound\/} on
the number of terms
needed to represent every sufficiently large number is (effectively) computable.
Now we just write first-order statements
that every sufficiently large integer is a sum of $h$ terms of the
Beatty sequence, and evaluate them using the algorithm of Theorem~\ref{decide}, until we find the smallest $h$ that works.  The algorithm
is then guaranteed to terminate by Nathanson's theorem.

Similarly we can write a first-order statement
that every natural number (not just every sufficiently large
natural number) is a sum of some number of terms.
\end{proof}

\begin{remark}
The same result applies to the generalized Beatty sequences of Allouche and Dekking
\cite{Allouche&Dekking:2019}.
\end{remark}

We now illustrate the ideas for a particular example.
We use the free software \texttt{Walnut}, which can prove theorems
about Ostrowski-automatic sequences, if they are stated in first-order
logic.   For more details about \texttt{Walnut}, see
\cite{Mousavi:2016,Shallit:2023}.

Let $\varphi = (1+\sqrt{5})/2$, the golden ratio, and consider the sequence
${\bf c} = (\lfloor n\varphi + {1 \over 2}\rfloor)_{n \geq 0}$.   This is sequence 
\seqnum{A007067} in the OEIS \cite{Sloane:2026}, and was originally studied by Johann Bernoulli \cite[pp.~66--69]{Venkov:1970}.  For convenience, we used
the Fibonacci numeration system, which is a feature of {\tt Walnut}.

\begin{theorem}
The sequence $\bf c$ is Fibonacci synchronized.
\label{varphisynch}
\end{theorem}

\begin{proof}
Starting from the known automaton (called {\tt phin}) for the Beatty sequence
$(\lfloor n \varphi \rfloor)_{n \geq 0}$
(see, e.g., \cite[\S 10.11]{Shallit:2023}), we can get
an automaton for $\bf c$ as follows (in {\tt Walnut})
\begin{verbatim}
reg shift {0,1} {0,1} "([0,0]|[0,1][1,1]*[1,0])*":
def phin "?msd_fib (s=0&n=0) | Ex $shift(n-1,x) & s=x+1":
def eta "?msd_fib Er $phin(2*n,r) & z=(r+1)/2":
\end{verbatim}
This produces the synchronized automaton for $\bf c$ in Figure~\ref{fig1}.  
The {\tt reg} command creates an automaton for a regular expression
that implements the assertion that the second argument is the left-shift
of the first argument.  
\begin{figure}[H]
\begin{center}
\includegraphics[width=6.5in]{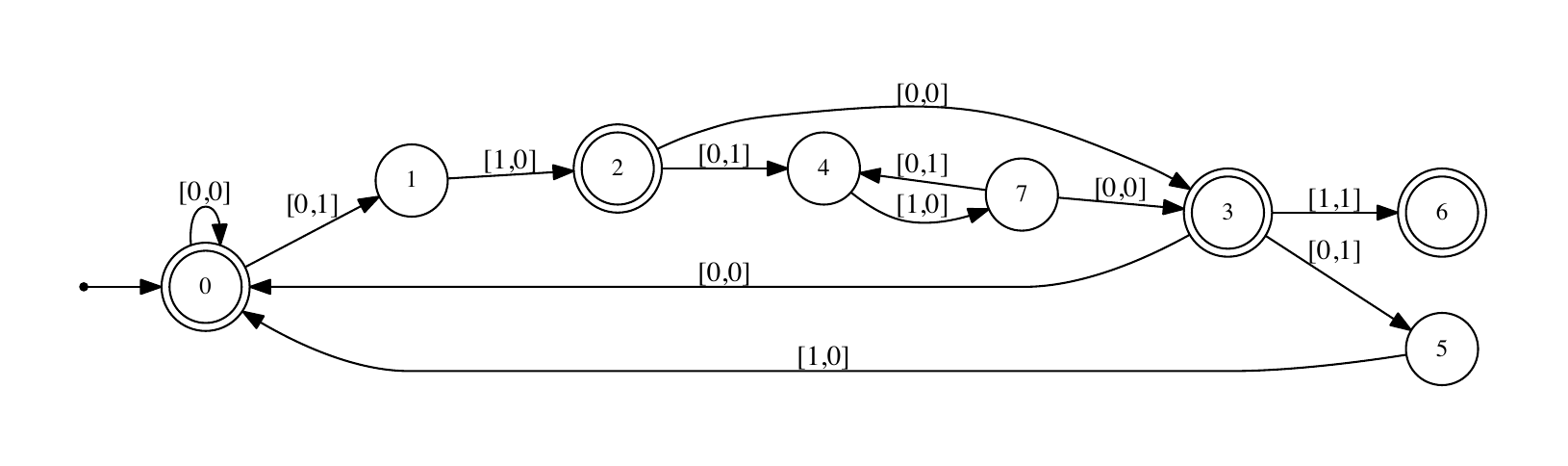}
\end{center}
\caption{Synchronized automaton for $(\lfloor n\varphi + {1\over 2} \rfloor)_{n \geq 0}$.}
\label{fig1}
\end{figure}
\end{proof}

\begin{theorem}
Every non-negative integer $n\not=1$ can be written as the sum of two elements
of $\bf c$.
\end{theorem}

\begin{proof}
We use the following {\tt Walnut} commands:
\begin{verbatim}
def iseta "?msd_fib En $eta(n,s)":
eval test "?msd_fib An (n!=1) <=> Ei, j $iseta(i) & $iseta(j) & n=i+j":
\end{verbatim}
and this returns {\tt TRUE}.   The result is proved.
\end{proof}

Using the same ideas we can prove the following.  Call two sets of
integers $S, T$ {\it asymptotically additively equivalent\/} if
there exists an integer $M$ such that every integer $n \geq M$
can be written as the sum of two elements of $S$ if and only if
the same is true of $T$.
\begin{theorem}
Given $\alpha_1, \alpha_2, \beta_1, \beta_2$ real numbers in the same
quadratic field, we can decide if 
the two inhomogeneous Beatty sequences
$(\lfloor \alpha_1 n + \beta_1 \rfloor)_{n \geq 1}$ and
$(\lfloor \alpha_2 n + \beta_2 \rfloor)_{n \geq 1}$ are
asymptotically additively equivalent.
\end{theorem}

The (simple) proof  is left to the reader.

\section{Proofs of some conjectures of Reble}
\label{reble}

As further applications we prove some conjectures of Don Reble regarding
the OEIS sequences \seqnum{A189377}, 
\seqnum{A189378}, and \seqnum{A189379}.

These conjectures relate to the existence
of monochromatic arithmetic progressions
of length $3$ in the Fibonacci word $\bf f$, defined by
the fixed point of the morphism $0 \rightarrow 01$, $1 \rightarrow 0$, or, alternatively, as the Sturmian word
corresponding to $\alpha = (3-\sqrt{5})/2$.
More specifically, we are interested in those $r$ for which
\begin{itemize}
\item[(i)] there exists $i$ such that
${\bf f}[i] = {\bf f}[i+r] = {\bf f}[i+2r] = 0$
\item[(ii)] there exists $j$ such that
${\bf f}[j] = {\bf f}[j+r] = {\bf f}[j+2r] = 1$.
\end{itemize}

Reble's conjectures can be stated as follows.
\begin{theorem}
Let $r = 2$, $s = (\sqrt{5}-1)/2$, and $t = (1+\sqrt{5})/2$.
Define
\begin{align*}
a(n) &= n + \lfloor ns/r \rfloor + \lfloor nt/r \rfloor \\
b(n) &= n + \lfloor nr/s \rfloor + \lfloor nt/s \rfloor \\
c(n) &= n + \lfloor nr/t \rfloor + \lfloor ns/t \rfloor .
\end{align*}
Then 
\begin{itemize}
\item[(a)]
there exists $m$ such that $m = a(n)$ iff (i) holds
for $r= n+1$ but (ii) does not for $r = n+1$;
\item[(b)]
there exists $m$ such that $m = b(n)$ iff neither
(i) nor (ii) holds for $r = n+1$;
\item[(c)]
there exists $m$ such that $m = c(n)$ iff both
(i) and (ii) hold for $r = n+1$.
\end{itemize}
\end{theorem}

\begin{proof}
We implement the formulas above in \texttt{Walnut}.   Here {\tt F} is {\tt Walnut}'s abbreviation for the Fibonacci word, indexed starting at position $0$,
and {\tt phin} is the synchronized automaton computing
$\lfloor \varphi n \rfloor$ for $\varphi = (1+\sqrt{5})/2$, the golden ratio,
given in the proof of Theorem~\ref{varphisynch}.
\begin{verbatim}
def three0 "?msd_fib Ei (F[i]=@0) & (F[i+n]=@0) & (F[i+2*n]=@0)":
def three1 "?msd_fib Ei (F[i]=@1) & (F[i+n]=@1) & (F[i+2*n]=@1)":

def fibsr "?msd_fib Er $phin(n,r) & s=(r-n)/2":
def fibtr "?msd_fib Er $phin(n,r) & s=r/2":
def fibrs "?msd_fib $phin(2*n,s)":
def fibts "?msd_fib Er $phin(n,r) & s=r+n":
def fibrt "?msd_fib Er $phin(2*n,r) & s=r-2*n":
def fibst "?msd_fib Er $phin(n,r) & s=2*n-(r+1)":

def a189377 "?msd_fib En,x,y $fibsr(n,x) & $fibtr(n,y) & z=n+x+y":
def a189378 "?msd_fib En,x,y $fibrs(n,x) & $fibts(n,y) & z=n+x+y":
def a189379 "?msd_fib En,x,y $fibrt(n,x) & $fibst(n,y) & z=n+x+y":

def reble1 "?msd_fib $three0(n) & ~$three1(n)":
def reble2 "?msd_fib ~$three0(n) & ~$three1(n)":
def reble3 "?msd_fib $three0(n) & $three1(n)":

eval rebleconj1 "?msd_fib An (n>=1) => ($reble1(n+1) <=> $a189377(n))":
eval rebleconj2 "?msd_fib An (n>=1) => ($reble2(n+1) <=> $a189378(n))":
eval rebleconj3 "?msd_fib An (n>=1) => ($reble3(n+1) <=> $a189379(n))":
\end{verbatim}
When we run them, we get 
\texttt{TRUE}
as the result.   The entire computation takes less than a second.
\end{proof}

\section{Graham's question}
\label{graham}

For real numbers $\alpha > 0$ and $\beta$, define the set
$S(\alpha, \beta) = \{ \lfloor \alpha n + \beta \rfloor 
\suchthat n \geq 1 \}$.
Graham \cite{Graham:1973} studied the following question:
for which finite sets of pairs of real numbers
$\{ (\alpha_i, \beta_i) \suchthat 1 \leq i \leq r \}$
does every sufficiently large integer belong to exactly
one $S(\alpha_i, \beta_i)$?  Such a finite set of pairs is called
an {\it eventual covering family\/} or ECF.

Our results prove that a special case is decidable:
\begin{theorem}
Suppose that all the $\alpha_i$ and $\beta_i$, $1 \leq i \leq r$,
lie in the
same quadratic field.  Then Graham's problem is decidable
for $\{ (\alpha_i, \beta_i) \suchthat 1 \leq i \leq r \}$.
\end{theorem}

\begin{proof}
We can create a first-order logical formula
asserting that there exists an $M$ such that, for all
$n \geq M$, we have
\begin{itemize}
\item[(i)]
there exist $k$ and at least one $i$, $1 \leq i \leq r$, such that
$n = \lfloor \alpha_i k + \beta_i \rfloor$;
\item[(ii)]
there do not exist $i, j, k_i, k_j$, $1 \leq i < j \leq r$, such that
$n = \lfloor \alpha_i k_i + \beta_i \rfloor = \lfloor \alpha_j k_j +
\beta_j \rfloor$.
\end{itemize}
It now follows from our approach that both of these are
decidable.
\end{proof}

\begin{example}
We show that 
$\{ (2\alpha, 0), (2\alpha, \alpha), (\alpha^2, 0) \}$
is such a family, where $\alpha = (1+\sqrt{5})/2$.
\begin{verbatim}
eval case_i "?msd_fib An (n>=2) => Ek k>=1 & ($phin(2*k,n) | $phin(2*k+1,n) 
   | $phi2n(k,n))":
eval case_ii "?msd_fib ~En,j,k j>=1 & k>=1 & (($phin(2*j,n) & $phin(2*k+1,n)) 
   | ($phin(2*j+1,n) & $phi2n(k,n)) | ($phi2n(j,n) & $phin(2*k,n)))":
\end{verbatim}
Both of these return {\tt TRUE}.
\end{example}

\section{A sequence of Hildebrand, Li, Li, and Xie}
\label{hildebrand}

Hildebrand, Li, Li, and Xie \cite{Hildebrand&Li&Li&Xie:2019} studied a certain sequence
$\tilde{c}(n)$, defined as follows:  let $a(n) = \lfloor n\varphi^4\rfloor$, $b(n) = \lfloor n \varphi^3 \rfloor$,
$c(n) = \lfloor n \varphi \rfloor$, 
and $\tilde{c}(n)$ be the $n$'th positive integer
that is not contained in the union
$\{ a(i) \suchthat i \geq 1 \} \, \cup \, 
\{ b(i) \suchthat i \geq 1 \}$.
They showed that $0 \leq c(n) - \tilde{c}(n) \leq 2$.

We can reprove this and do a little more, using our techniques:
\begin{theorem}
The Fibonacci automaton displayed in Figure~\ref{fig3}, on input $n$ in the Fibonacci numeration system, computes
$c(n) - \tilde{c}(n)$.
\begin{figure}[H]
\begin{center}
\includegraphics[width=6.5in]{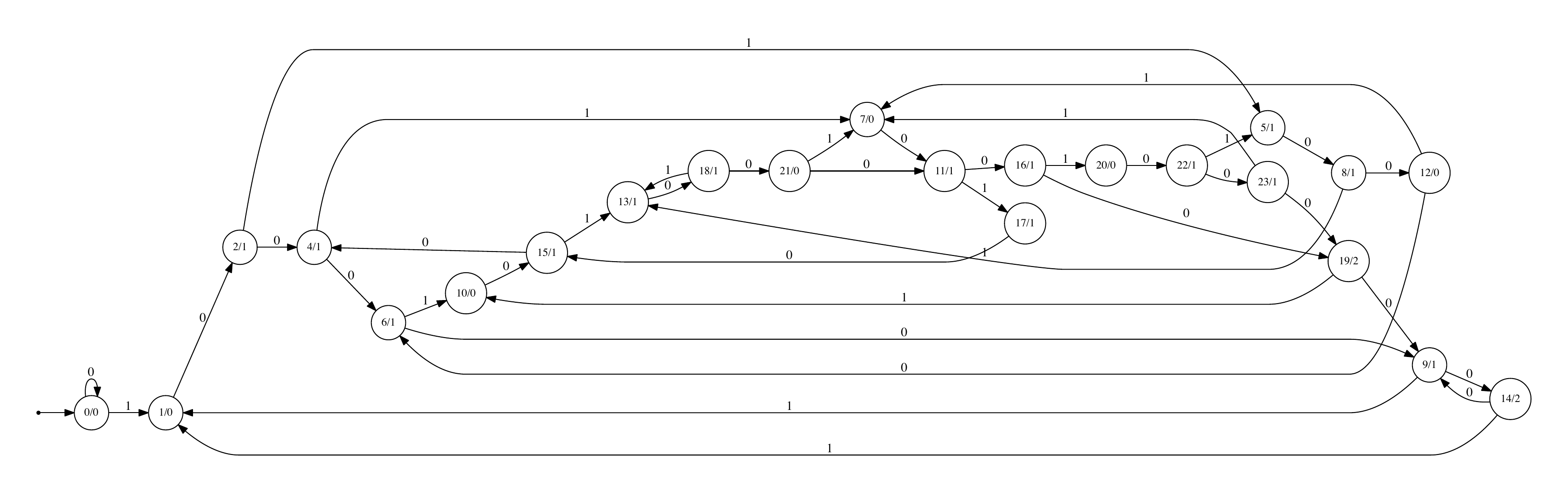}
\end{center}
\caption{Fibonacci automaton for $c(n) - \tilde{c}(n)$. A state labeled $q/a$ has output $a$.}
\label{fig3}
\end{figure}
\end{theorem}

\begin{proof}
We use the fact that $\varphi^4 = 3\varphi + 2$ and $\varphi^3 = 2\varphi + 1$.  We can therefore define synchronized automata for
$a(n)$, $b(n)$, and $c(n)$, as follows:
\begin{verbatim}
def a "?msd_fib Em $phin(3*n,m) & z=m+2*n":
def b "?msd_fib Em $phin(2*n,m) & z=m+n":
def c "?msd_fib $phin(n,z)":
\end{verbatim}

Now $\tilde{c}(n)$ is defined as the $n$'th positive integer not contained in the union $\{ a(i) \suchthat i \geq 1 \} \, \cup \, 
\{ b(i) \suchthat i \geq 1 \}$.
First, let us check that there is no intersection between the images
of $a$ and $b$:
\begin{verbatim}
eval no_inter "?msd_fib ~Em,n,x m>1 & n>1 & $a(m,x) & $b(n,x)":
\end{verbatim}
And {\tt Walnut} returns {\tt TRUE}.

Next, let's define inverses to $a(n)$ and $b(n)$ as follows:  
we choose $a^{-1} (m)$ such that it is the least value for
which $a(a^{-1} (m)) \leq m < a(a^{-1}(m) + 1)$, and the
analogous inequality for $b^{-1} (n)$:
\begin{verbatim}
def ai "?msd_fib  Et,u $a(z,t) & $a(z+1,u) & t<=m & m<u":
# ai(m) = z, t=a(ai(m)), u=a(ai(m)+1)
def ainv "?msd_fib $ai(m,z) & Ax (x<z & x>0) => ~$ai(m,x)":
# the least m with the property

def bi "?msd_fib  Et,u $b(z,t) & $b(z+1,u) & t<=m & m<u":
# bi(m) = z, t=b(bi(m)), u=b(bi(m)+1)
def binv "?msd_fib $bi(m,z) & Ax (x<z & x>0) => ~$bi(m,x)":
# the least m with the property
\end{verbatim}

Next, we can define ${\tilde{c}}(n)$ as the least $y$
such that $a^{-1} (y) + b^{-1}(y) + n = y$:
\begin{verbatim}
def ctw "?msd_fib Es,t $ainv(y,s) & $binv(y,t) & s+t+n=y":
def ctwid "?msd_fib $ctw(n,y) & Az (z<y) => ~$ctw(n,z)":
\end{verbatim}

Next we can check that $\tilde{c}(n)$ is indeed the complementary
sequence.
\begin{verbatim}
eval check4 "?msd_fib Ax (x>0) => (En (n>=1) & ($a(n,x)|$b(n,x))) <=> 
  (~En  (n>=1) & $ctwid(n,x))":
# ctwid is the complementary sequence of a(n) union b(n)
\end{verbatim}
And {\tt Walnut} returns {\tt TRUE}.

Next, we check that the inequality $0 \leq c(n) - \tilde{c}(n) \leq 2$ holds:
\begin{verbatim}
eval check5 "?msd_fib An,x,y ($phin(n,x) & $ctwid(n,y)) => (x=y|x=y+1|x=y+2)":
\end{verbatim}
and {\tt Walnut} returns {\tt TRUE}.

Finally, we combine the individual automata for the three
possible values of $c(n) - \tilde{c}(n)$ to get a single
automaton that computes the difference:
\begin{verbatim}
def diff0 "?msd_fib Ex $phin(n,x) & $ctwid(n,x)":
def diff1 "?msd_fib Ex $phin(n,x+1) & $ctwid(n,x)":
def diff2 "?msd_fib Ex $phin(n,x+2) & $ctwid(n,x)":
combine diff diff1=1 diff2=2:
\end{verbatim}
which gives us the automaton in Figure~\ref{fig3}. 
\end{proof}

\section{The Kimberling swappage problem}
\label{kimberling}

Let $L(n) = \lfloor n \varphi \rfloor$ be the lower
Wythoff sequence and $U(n) = \lfloor n \varphi^2 \rfloor$
be the upper Wythoff sequence.   Kimberling defined
the ``lower even swappage sequence'' $V_{\rm le}(n)$ as follows:  if $U(n)$ is 
even, then $V_{\rm le}(n) := U(n)$.   Otherwise, if $U(n)$ is odd,
then $V_{\rm le}(n) := L(i)$, where $i$ is the least index such that
$U(n-1) < L(i) < U(n+1)$ and
$L(i-1) < U(n) < L(i+1)$.  Here $(V_{\rm le}(n))$ is OEIS sequence
\seqnum{A141104}.  He conjectured
that the sequence $S(n):= V_{\rm le}(n)/2$ consists of those
positive integers not in the Beatty sequence
$(\lfloor n \varphi^3 \rfloor)_{n \geq 1}$, which was
later proved by Russo
and Schwiebert \cite{Russo&Schwiebert:2011}.

We can use {\tt Walnut} to provide an alternate proof
of Kimberling's conjecture.   The ideas are those similar to those of the previous section.   First, using the ideas
in \cite[\S 10.15]{Shallit:2023}, we guess a synchronized automaton {\tt leswap} for $V_{\rm le}(n)$ (it has $34$ states), and we then
verify it satisfies the definition.  Here {\tt phi2n} is a synchronized
automaton computing $\lfloor \varphi^2 n \rfloor$, taken from 
\cite[\S 10.11]{Shallit:2023}.
\begin{verbatim}
def chk1 "?msd_fib An Ex $leswap(n,x)":
def chk2 "?msd_fib An ~Ex1,x2 x1!=x2 & $leswap(n,x1) & $swap(n,x2)":

def odd "?msd_fib Ex n=2*x+1":
def even "?msd_fib Ex n=2*x":
def index "?msd_fib Aa,b,c,d,e,f ($phin(i-1,a) & $phin(i,b) & $phin(i+1,c) & 
   $phi2n(n-1,d) & $phi2n(n,e) & $phi2n(n+1,f)) => (a<e & e<c & d<b & b<f)":
def leastindex "?msd_fib $index(i,n) & Aj (1<=j & j<i) => ~$index(j,n)":
eval checkeven "?msd_fib An,x,y ($phi2n(n,x) & $even(x) & $leswap(n,y)) 
   => x=y":
eval checkodd "?msd_fib Ai,n,x,y,z (i>=1 & n>=1 & $phi2n(n,x) & $odd(x) & 
   $leswap(n,y) & $leastindex(i,n) & $phin(i,z)) => y=z":
\end{verbatim}
The two results of {\tt TRUE} from the last two commands verify that our guessed sequence
satisfies the definition.

Next, we verify the characterization in terms of the
Beatty sequence:
\begin{verbatim}
def phi3n "?msd_fib Ey $phin(2*n,y) & z=y+n":
eval kimber "?msd_fib Ax (x>0) => ((En (n>=1) & $phi3n(n,x)) 
   <=> (~En,y (n>=1) & $leswap(n,y) & 2*x=y))":
\end{verbatim}
and {\tt Walnut} returns {\tt TRUE}.

The same techniques can be used to prove the  conjectures stated  in the OEIS
about analogous sequences, namely
\begin{itemize}
    \item the ``upper even'' sequence $V_{\rm ue}(n)$ \seqnum{A141105};
    \item the ``lower odd'' sequence $V_{\rm lo}(n)$
    \seqnum{A141106}; and
    \item the ``upper odd'' sequence $V_{\rm uo}(n)$
    \seqnum{A141107}.
\end{itemize}
In particular, we can prove the following theorem with
{\tt Walnut}:
\begin{theorem}
    We have
    \begin{itemize}
        \item[(a)] $V_{\rm ue} (n) = 2 \lfloor {{\varphi+1}\over 2} n + {1 \over 2} \rfloor$;
         \item[(b)] $V_{\rm le} (n) = 2 \lfloor {{\varphi+1}\over 2} n  \rfloor$;
         \item[(c)] $V_{\rm uo} (n) = 2 \lfloor {{\varphi+1}\over 2} n  \rfloor + 1$;
         \item[(d)] $V_{\rm lo} (n) = 2 \lfloor {{\varphi+1}\over 2} n + {1 \over 2}  \rfloor- 1$.
    \end{itemize}
\end{theorem}

We omit the details.

\section{Sums-complement sets}
\label{sec11}
Clark Kimberling has studied (see \seqnum{A276871} in the OEIS) what he calls ``sums-complement" sets for irrational numbers $\alpha$.  These are the positive integers
that cannot be written as the difference of
two elements of the Beatty sequence 
$\lfloor n\alpha  \rfloor$.   More precisely,
$$S_{\alpha} := \{ n \geq 1 \suchthat \neg\exists i,j \ 
n = \lfloor i \alpha  \rfloor - \lfloor j \alpha \rfloor \}.$$

A set $S$ is said to be $\gamma$-Ostrowski automatic if there is a finite automaton accepting the $\gamma$-Ostrowski representation of $n$ as input and accepting if and only if $n \in S$.   Then we have the following proposition.
\begin{theorem}
 Suppose $0 < \gamma < 1$ is a quadratic irrational
such that $1 \over\gamma$ has purely periodic continued
fraction
$$[a_1, a_2, \ldots, a_P, a_1, a_2, \ldots, a_P, a_1, a_2, \ldots, a_P,
\ldots],$$
and suppose $\alpha \in \Que(\gamma)$
such that $\alpha \geq 0$.   Then the sums-complement
set $S_{\alpha}$ is $\gamma$-Ostrowski automatic.
\end{theorem}

\begin{proof}
The claim that $n \in S_{\alpha}$ is a first-order statement in terms of the Beatty sequence for $\alpha$.
\end{proof}

As a consequence, we may use {\tt Walnut} to obtain
explicit Ostrowski automata for the quadratic
numbers tabulated in \seqnum{A276871}.

\begin{example}
    Let $\gamma = [0,\overline{4,1,1,1}] = (\sqrt{7}-2)/3$
and $\alpha = \sqrt{7} = 3\gamma + 2$.   Then there is an
$\gamma$-Ostrowski automaton of 6961 states that recognizes the sums-complement set $S_{\alpha}$, which is
sequence \seqnum{A276873} in the OEIS.  This automaton can be computed with the following {\tt Walnut} code:
\begin{verbatim}
ost sqrt7 [0] [4 1 1 1];
def beatty7 "?msd_sqrt7 Eu,v n=u+1 & $shift4(u,v) & v=9*z+14*u"::
# largest intermediate automaton has 1710130 states
# time is 85648 ms
# 65 states

def beat7 "?msd_sqrt7 Ex $beatty7(3*n,x) & z=x+2*n":
# 96 states, 164ms

def a276873 "?msd_sqrt7 ~Em,n,x,y m>=1 & n>=1 & $beat7(m,x) 
   & $beat7(n,y) & z+x=y"::
# largest intermediate automaton has 64815 states
# 6961 states, 14150 ms
\end{verbatim}
\end{example}

\section{Sequences related to \texorpdfstring{$\sqrt{2}$}{sqrt(2)}}
\label{sec12}

In this section we find Ostrowski automata
for seven sequences from the OEIS related to $\sqrt{2}$, and we re-prove a number of conjectures related to these sequences.
The first few terms are given below in
Table~\ref{sqrt2}.
\begin{table}[H]
    \centering
    \begin{tabular}{c|ccccccccccccccccccccccc}
     $n$ &   0& 1& 2& 3& 4& 5& 6& 7& 8& 9&10&11&12&13&14&15&16\\
     \hline
\seqnum{A097508}$(n)$ & 0& 0& 0& 1& 1& 2& 2& 2& 3& 3& 4& 4& 4& 5& 5& 6 & 6\\
 \seqnum{A001951}$(n)$ & 0& 1& 2& 4& 5& 7& 8& 9&11&12&14&15&16&18&19&21&22\\
 \seqnum{A003151}$(n)$ & 0& 2& 4& 7& 9&12&14&16&19&21&24&26&28&31&33&36&38\\
\seqnum{A276862}$(n)$ & 2& 2& 3& 2& 3& 2& 2& 3& 2& 3& 2& 2& 3& 2& 3& 2& 3\\
\seqnum{A097509}$(n)$ & 3& 2& 3& 2& 3& 2& 2& 3& 2& 3& 2& 2& 3& 2& 3& 2& 3 \\
\seqnum{A080754}$(n)$ & 0& 3& 5& 8&10&13&15&17&20&22&25&27&29&32&34&37&39\\
\seqnum{A082844}$(n)$ & 2&3&2&3&2&2&3&2&3&2&2&3&2&3&2&3&2\\
    \end{tabular}
    \caption{Seven OEIS sequences related to $\sqrt{2}$.}
    \label{sqrt2}
\end{table}

We start with \seqnum{A097508}$(n)$, whose definition is $\lfloor n(\sqrt{2} - 1) \rfloor$.
In the corresponding Ostrowski numeration system, the digits can be 0, 1 or 2. The shifting automaton {\tt shift1} can be built using Corollary~\ref{five}, and it is depicted in Figure~\ref{tm3}.
\begin{figure}[H]
\begin{center}
\includegraphics[width=4.5in]{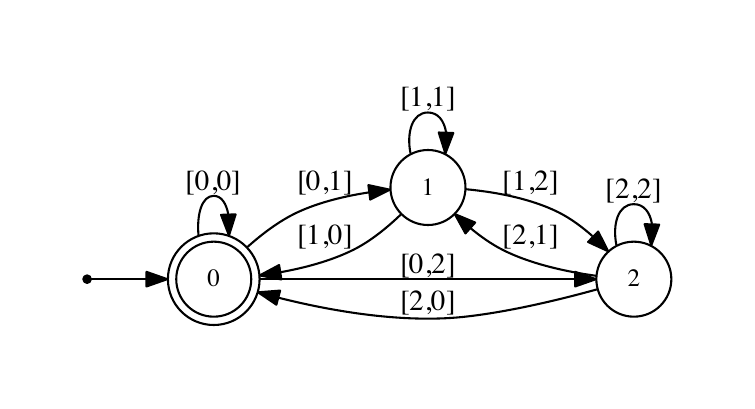}
\end{center}
\caption{Shift automaton.}
\label{tm3}
\end{figure}
Using Theorem~\ref{main}, we can obtain the synchronized automaton for $\lfloor n(\sqrt{2} - 1) \rfloor$ from the following {\tt Walnut} commands. 
\begin{verbatim} 
ost s2 [0] [2]:
# define the numeration system
def a097508 "?msd_s2 (n=0 & z=0) | (Eu,v n=u+1 & $shift1(u,v) & v=z+2*u)":
\end{verbatim}

Next we turn to OEIS sequence
\seqnum{A001951}$(n)$, defined by
$\lfloor n \sqrt{2} \rfloor$.
We apply Theorem~\ref{main1} using $a=1$, $b=1$, $c=1$, $d=0$, and $e=0$ to create the automaton for \seqnum{A001951} using the following {\tt Walnut} command.
\begin{verbatim} 
def a001951 "?msd_s2 Eu $a097508(n,u) & z=u+n":
\end{verbatim}

Next we turn to OEIS sequence
\seqnum{A003151}$(n)$, defined by $\lfloor n (\sqrt{2} + 1) \rfloor$.  We apply Theorem~\ref{main1} using $a=2$, $b=1$, $c=1$, $d=0$, and $e=0$ to create the automaton for \seqnum{A003151} using the following {\tt Walnut} command.
\begin{verbatim} 
def a003151 "?msd_s2 Eu $a097508(n,u) & z=u+2*n":
\end{verbatim}

Next we look at OEIS sequence
\seqnum{A276862}$(n)$, defined to be the
first difference of \seqnum{A003151}.  We can create a synchronized automaton for it
as follows:
\begin{verbatim}
def a276862 "?msd_s2 Eu,v $a003151(n,u) & $a003151(n+1,v) & z+u=v":
\end{verbatim}

Next we look at OEIS sequence \seqnum{A097509}, defined to be the number of times that $n$ occurs as a value of 
\seqnum{A097508}.  It is clear that this number is either $2$ or $3$, so we can build the synchronized automaton as follows:
\begin{verbatim}
def three_times "?msd_s2 Eu,v,w u<v & v<w & $a097508(u,n) & $a097508(v,n) 
   & $a097508(w,n)":
def a097509 "?msd_s2 (z=3 & $three_times(n)) | (z=2 & ~$three_times(n))":
\end{verbatim} 

Michel Dekking conjectured in 2020 that
$\seqnum{A097509}(n) = \seqnum{A276862}(n)$ for $n \geq 1$ (see the first comment
in \seqnum{A276862}).   We can prove this with the following {\tt Walnut} code:
\begin{verbatim}
eval dek "?msd_s2 An,x (n>=1) => ($a097509(n,x) <=> $a276862(n,x))":
\end{verbatim}
And {\tt Walnut} returns {\tt TRUE}.

Next we turn to OEIS sequence \seqnum{A080754}$(n)$, which is defined
to be $\lceil n (\sqrt{2} + 1) \rceil$.
\begin{verbatim}
def a080754 "?msd_s2 (n=0 & z=0) | (n>0 & Eu $a003151(n,u) & z=u+1)":
\end{verbatim}

We can now find a relationship between
\seqnum{A097509} and \seqnum{A080754};
namely,
\begin{equation}
\seqnum{A097509}(i) =  \seqnum{A080754}(i+1) - \seqnum{A080754}(i)
\label{fd1}
\end{equation}
for all $i \geq 0$.
\begin{verbatim}
eval check_equality "?msd_s2 Ai,u,v,w ($a080754(i+1,u) & $a080754(i,v) 
   & $a097509(i,w)) => w+v=u":
\end{verbatim}
And {\tt Walnut} returns {\tt TRUE}.

Finally, we consider $a(n) := \seqnum{A082844}(n)$.
It has a rather complicated definition, as follows:
\begin{equation}
a(t) = \begin{cases}
    3, & \text{if $t = 1$}; \\
    2, & \text{if $t = 0,2$}; \\
    a(n), & \text{if $t = a(1)+\cdots+a(n)$};\\
    2, & \text{if $a(t-1)=3$ and $\neg\exists n\geq 0 \ \ t=a(1)+\cdots+a(n)$};\\
    3, & \text{if $a(t-1)=2$ and $\neg\exists n\geq 0 \ \  t=a(1)+\cdots+a(n)$}.
\end{cases}
\label{defa}
\end{equation}
To get a formula for $a(n)$, we use a somewhat roundabout approach.  Namely, we
{\it define\/} the sequence $b(n)$ as $ \seqnum{A097509}(n+1)$ for $n \geq 0$,
and then show $b(n)$ satisfies the
recurrence Eq.~\eqref{defa}.   This will show that in fact $a(n) = b(n)$ for all
$n$.

This approach requires computing
$b(1)+\cdots + b(n)$, which we do as follows:
\begin{align*}
\sum_{1 \leq i \leq n} b(i) 
&= \sum_{1 \leq i \leq n} \seqnum{A097509}(i+1) \\
&= \sum_{1 \leq i \leq n} 
\left( \seqnum{A080754}(i+2) - \seqnum{A080754}(i+1) \right) \\
&= \seqnum{A080754}(n+2) - \seqnum{A080754}(2) \\
&= \seqnum{A080754}(n+2) -5.
\end{align*}
We then use the following {\tt Walnut} code:
\begin{verbatim}
def b2 "?msd_s2 $a097509(n+1,z)":
eval check1 "?msd_s2 An,u,v,w ($a080754(n+2,u) & $b2(u-5,v) & $b2(n,w)) 
   => v=w":
eval check2 "?msd_s2 At ($b2(t-1,3) & ~En $a080754(n+2,t+5)) => $b2(t,2)":
eval check3 "?msd_s2 At ($b2(t-1,2) & ~En $a080754(n+2,t+5)) => $b2(t,3)":
   \end{verbatim}
{\tt Walnut} returns {\tt TRUE} for all cases, which proves the equality of $a(n)$ and $b(n)$.

This also proves the conjecture of Benedict Irwin
from 2016, to the effect that
$$\seqnum{A082844}(n) = \seqnum{A097509}(n+1)$$ for all
$n \geq 0$ (see the comments to
\seqnum{A082844}).  
For more information, see 
\cite[2nd solution to problem B6]{Bhargava&Kedlaya&Ng:2021}.

\section{Fractional parts}
\label{fractional}

Let $\{ x \}$ denote the fractional part of $x$, that is, ${x \bmod 1} = x - \lfloor x \rfloor$.
Inspired by recent work of Khani, Valizadeh, and
Zarei \cite{Khani&Valizadeh&Zarei:2025},
we show how to handle
assertions like
$\lbrace \alpha x + \beta \rbrace < \lbrace \alpha' y + \beta' \rbrace$
in {\tt Walnut}, where $x$ and $y$ are natural number variables
and $\alpha, \alpha', \beta, \beta'$ all lie in the same quadratic field.
The key idea is the next very simple proposition.
\begin{proposition}
Let $x, y$ be real numbers.  Then 
$\{x\} < \{y\}$ if and only if $\lfloor x - y \rfloor < \lfloor x \rfloor - \lfloor y \rfloor$.
\label{prop8}
\end{proposition}

\begin{proof}
Write $x = \lfloor x \rfloor + \{ x \}$, $y = \lfloor y \rfloor + \{y \}$.
Then 
\begin{align*}
\{ x \} < \{ y \} & \iff x - \lfloor x \rfloor < y - \lfloor y \rfloor \\
& \iff  x-y < \lfloor x \rfloor - \lfloor y \rfloor \\
& \iff  \lfloor x-y \rfloor < \lfloor x \rfloor - \lfloor y \rfloor .
\end{align*}
\end{proof}

The virtue of Proposition~\ref{prop8} is that we can turn a first-order
logic statement involving comparison of the sizes of fractional parts
into a first-order statement involving floors.

\begin{example}
Let us create an automaton to decide if $\lbrace \alpha x \rbrace <
\lbrace \alpha y \rbrace$, where $\alpha = (1+\sqrt{5})/2$, the
golden ratio.  Once more we use the
automaton {\tt phin} that
accepts the Zeckendorf representation of the natural number pair
$(x,y)$ in parallel
if and only if $y = \lfloor \alpha x \rfloor$.

Translation of Proposition~\ref{prop8} into {\tt Walnut} involves a tiny
bit of extra effort, because the default domain for integers is $\Enn$ and not
$\Zee$.  Thus we need two cases, depending on whether $x > y$ or
$x < y$.

\begin{verbatim}
reg shift {0,1} {0,1} "([0,0]|[0,1][1,1]*[1,0])*":
def phin "?msd_fib (x=0 & y=0) | Ez $shift(x-1,z) & y=z+1":
def cfp "?msd_fib Er,s,t $phin(x,r) & $phin(y,s) & $phin(x-y,t) & t+s<r":
# decides if {alpha*x} < {alpha*y} provided x>y
def compare "?msd_fib (x>y & $cfp(x,y)) | (y>x & ~$cfp(y,x))":
\end{verbatim}
This generates an automaton {\tt compare} of 22 states.  The inputs
are $x$ and $y$ in Zeckendorf representation, most-significant-digit
first, and the input is accepted if and only if $\{ \alpha x \} <
\{ \alpha y \}$.

This implementation works for $x, y \geq 0$.  But we might want to also be able
to compare $\{ \alpha x \}$ and $\{ \alpha y \}$ where one or both of $x,y$
is negative.  We can do this using Proposition~\ref{prop8} together with
the identity $-\lfloor -x \rfloor = \lceil x \rceil$, and the fact that
$\lceil x \rceil = \lfloor x \rfloor + 1$ for $x \not\in \Zee$.
\end{example}

\begin{example}
Similarly we can check if $\{ \alpha x \} < c/d$ for quadratic
$\alpha$ and fixed natural number
constants $0 \leq c < d$.  By Proposition~\ref{prop8}
and Lemma~\ref{simple}
we have 
$$\{ \alpha x \} < c/d \iff  \left\lfloor {{{\lfloor \alpha d x \rfloor} - c} 
\over {d}} \right\rfloor < \lfloor \alpha x \rfloor .$$
For example, for $\alpha = (1+\sqrt{5})/2$ and $c=1$, $d = 2$,
we get an automaton to check if
$\{\alpha x \} < 1/2$, as follows:
\begin{verbatim}
def cmphalf "?msd_fib (x=0) | Ey,z $phin(2*x,y) & $phin(x,z) & (y-1)/2 < z":
\end{verbatim}
This gives an $8$-state automaton to decide if $\{ \alpha x \} < 1/2$.
\end{example}

Dan Asimov published a query on the math-fun mailing
list on September 23 2025, which we summarize as follows:  suppose $\alpha$ is an irrational number,
and $x,y,z$ are three integers.  Mapping the half-open interval $[0,1)$
to the unit circle in the standard way via $z \rightarrow \exp(2\pi i z)$,
how do we determine the relative order of the fractional parts
$\{ \alpha x \}$,
$\{ \alpha y \}$,
$\{ \alpha z \}$ on the circle, clockwise or counter-clockwise?  He was
particularly interested in the case $\alpha = \sqrt{2}$.

We can do this for any quadratic irrational using the ideas of this section,
because the relative order depends completely on the three pairwise comparisons
of the fractional parts of $\{ \alpha x \}$,
$\{ \alpha y \}$, and
$\{ \alpha z \}$.

We now illustrate this idea for $\sqrt{2}$.  From Section~\ref{sec12}
we have an automaton {\tt a001951} that computes the map
$n \rightarrow \lfloor \sqrt{2} n \rfloor$.
\begin{verbatim}
def cfp2 "?msd_s2 Er,s,t $a001951(x,r) & $a001951(y,s) & $a001951(x-y,t) & t+s<r":
def compare2 "?msd_s2 (x>y & $cfp2(x,y)) | (y>x & ~$cfp2(y,x))":
def order2 "?msd_s2 $compare2(x,y) & $compare2(y,z)":
def ccw "?msd_s2 $order2(x,y,z) | $order2(y,z,x) | $order2(z,x,y)":
\end{verbatim}

This gives us a $44$-state automaton {\tt ccw} that solves the problem,
accepting if the ordering for input
$(x,y,z)$ is counter-clockwise and rejecting otherwise.

\section{The converse to Theorem~\ref{main1}}
\label{converse-sec}

In this section we prove the converse to Theorem~\ref{main1}.

We need a lemma.   Recall that the subword complexity (aka factor complexity) of a word is the function
$\rho(n)$ counting the number of distinct length-$n$ factors.
\begin{lemma}
Suppose $0 <\alpha_i, \beta_i < 1$
are real numbers for $1 \leq i \leq k$.  Let $${\bf s}_i (n) = 
\lfloor (n+1) \alpha_i + \beta_i \rfloor - \lfloor n \alpha_i + \beta_i \rfloor$$
be the associated Sturmian word
with slope $\alpha_i$ and
intercept $\beta_i$.
Form the word $\bf s$ over the alphabet
$\{ 0,1 \}^k$ by
${\bf s}[n] = ({\bf s}_1 [n],\ldots,
{\bf s}_k [n])$.  If the
$k+1$ numbers
$$1, \alpha_1, \ldots, \alpha_k$$
are linearly independent
over $\Que$, then
$\bf s$ has subword complexity
$(n+1)^k$.
\end{lemma}

\begin{proof}
By the multidimensional version of Kronecker's theorem (see, e.g., 
\cite[Thm.~442]{Hardy&Wright:1985}), 
the linear independence criterion implies
that for $\epsilon>0$ and all choices of the $b_i$,
there exist integers $q$ and $p_i$, $1 \leq i\leq k$ such that
$|\sum_{1\leq i\leq k} q \alpha_i - p_i + b_i| < \epsilon$.

It is known that, for all $n$, there exists a partition of $[0,1)$ into
$n+1$ intervals such that the length-$n$ factor of a Sturmian word starting at position $q$ is completely determined by 
which interval
$q \alpha_i +\beta_i$ lies in, modulo 1 \cite{Knuth:1972}.   In other words it is completely determined
by which of $n+1$ (circular) intervals $q\alpha_i$ lies in. Pick a particular length-$n$ factor in each of the $k$ Sturmian words.   For all $i$, $1 \leq i\leq k$, choose a point $b_i$ in the
corresponding
intervals that is at least $\epsilon$ away from the endpoints of the intervals, and choose $\epsilon$ sufficiently small.
Then by Kronecker's theorem some positive integer $ q$ exists that makes $q \alpha_i\bmod 1$ at distance at most $\epsilon$
from $b_i$, so it lies in the interior of the specified interval, so the $i$'th Sturmian word
beginning at position $q$ is the desired one.     Since this is true for each $i$, $1 \leq i \leq k$,
all 
$(n+1)^k$ possibilities for a factor of $\bf s$ must occur. 
\end{proof}

We can now prove the converse to Theorem~\ref{main1}.

\begin{theorem}
Suppose $0 < \gamma < 1$ is a quadratic irrational such that
$1/\gamma$ has a purely periodic continued fraction expansion.  
  Suppose the Beatty sequence
  $(\lfloor n \alpha + \beta \rfloor)_{n \geq 1}$ is $\gamma$-Ostrowski synchronized.  Then
$\alpha$ and $\beta$ both belong to  $\Que(\gamma)$.
Furthermore, the expression of both $\alpha$ and $\beta$
in terms of $\gamma$ is computable.
\label{main1-converse}
\end{theorem}

\begin{proof}
Suppose 
  $(\lfloor n \alpha + \beta \rfloor)_{n \geq 1}$ is $\gamma$-Ostrowski synchronized.
Our first goal is to prove that $\alpha \in \Que(\gamma)$.
Suppose $1,\alpha,\gamma$ are
linearly independent over $\Que$.
If the Beatty sequence
  $(\lfloor n \alpha + \beta \rfloor)_{n \geq 1}$ is $\gamma$-Ostrowski synchronized, then
the corresponding Sturmian sequence
(possibly recoded over the alphabet
$\{0,1\}$ if necessary) is $\gamma$-Ostrowski automatic.  
This is also true for the Sturmian
sequence associated with slope $\gamma$ and intercept $0$.   Then the product of these two sequences
(over the alphabet $\{0,1\}^2$)
is also $\gamma$-Ostrowski automatic.   Hence this product
has subword complexity $\Theta(n^2)$, but 
from Corollary~\ref{four},
we know that the subword complexity of every automatic sequence is $O(n)$.  

This contradiction shows that $1,\alpha,\gamma$ are linearly dependent over $\Que$, and hence $\alpha \in \Que(\gamma)$.  Therefore there are
integers $b,c,d$ with $d \geq 1$ such that that
$\alpha = (b\gamma + c)/d$.

We now show that we can computably determine $b,c$, and $d$.  To do so we iterate over all possible integer triples by increasing order of $|b|+|c|+d$.
For each triple $(b,c,d)$ we construct an $\gamma$-Ostrowski automaton for the Sturmian word with slope $(b\gamma+c)/d$ and intercept $0$.   We then check 
(using the logical approach) whether the set of factors occurring in Sturmian word is the same as that for the Sturmian word with slope $\alpha$ and intercept $\beta$.  Since all Sturmian words with the same slope have the same set of factors, we eventually find a suitable triple.  
Then $\alpha = (b\gamma+c)/d$ with $b, c, d$ known integers.  

We now turn to discussing $\beta$.
If $\beta = 0$ there is nothing more to prove.
Otherwise, without loss of generality, we may assume that
$0 < \beta < 1$.  
Put $x = n \alpha$ and $y = -\beta$ in
Proposition~\ref{prop8}.
Then we get that $\{n \alpha \} < \{-\beta \}$ if and only if
\begin{align}
\lfloor n \alpha+\beta \rfloor &= \lfloor n \alpha \rfloor \nonumber\\
&= \left\lfloor {{ bn \gamma+cn} \over d } \right\rfloor \label{gammaoa}\\
&= \left\lfloor {{ \lfloor bn \gamma \rfloor + cn} \over d} \right\rfloor .
\nonumber
\end{align}

By our hypothesis there is a computable
$\gamma$-Ostrowski automaton $A_1$ taking $n\geq 1$
as input and computing $\lfloor n \alpha + \beta \rfloor$.
Similarly, since we know $b,c,d$ we can construct
a $\gamma$-Ostrowski automaton taking $n\geq 1$
as input and computing
$$ \left\lfloor {{ \lfloor bn \gamma \rfloor + cn} \over d} \right\rfloor.$$
It follows that there is an automaton $A_2$ that takes $n$ as input
and checks if $\{n \alpha \} < \{-\beta \}$.

Now make a new automaton $A_3$ that accepts those integers $m \geq 1$
such that $m \equiv \modd{0} {d}$ and
\begin{equation}
m \equiv \modd{0} {d} \quad \text{ and } \quad
\{m \alpha \} < \{-\beta \} .
\label{crit}
\end{equation}
and further there is no $m'$, $0 < m' < m$, satisfying Eq.~\eqref{crit} such
that $\{m \alpha \} < \{m' \alpha\} $.
Roughly speaking, $A_3$
accepts those $m \equiv \modd{0} {d}$
such that $\{m \alpha \}$ consists of better and
better approximations to $\{ - \beta \}$.
Again by Kronecker's theorem the automaton $A_3$
accepts infinitely many $m$, and further the limit of
$\{m \alpha \}$ for $m$ accepted is $\{ - \beta \}$.
Define $A_4$ to be the automaton that on input $n$
multiplies $n$ by the constant $d$ and then calls
automaton $A_3$ on the result.   Thus $A_4$
also accepts infinitely many $n$ and the limit of
$\{ dn \alpha \} = \{ nb\gamma + cn \} = \{n b \gamma \}$
 for $n$ accepted is still $\{ -\beta \}$.
Now define $A_5$ so it accepts if its input $r$
satisfies $r \equiv \modd{0} {b}$ and $A_4$ accepts
$\lfloor r/b \rfloor$.  Thus $A_5$ accepts the Ostrowski representations
of an infinite sequence
of natural numbers $r_i$ such that $\{ r_i \gamma \}$
tends to $\{ -\beta \}$.  Hence, by the pumping lemma for
regular languages, there exist finite words $x', y', z'$, with $y'$ nonempty,
such that
$A_5$ accepts $x' {y'}^i z'$ for all $i \geq 0$.
Define $x = {z'}^R$, $y = {y'}^R$, $z = {x'}^R$, where the exponent of $R$
denotes reversal.

Now we need
an explicit expression for $\{ m \gamma \}$
in terms of the continued fraction $[a_0, a_1, \ldots]$ for an arbitrary
real number $\gamma$, using two basic results on continued
fractions that can be found, for example, in 
\cite[Chap.~1]{Rockett&Szusz:1992}.   First,
if $p_i/q_i$ is the $i$'th convergent, we have
$q_i \gamma - p_i = 
(-1)^i \gamma_0 \gamma_1 \cdots \gamma_i$, where 
$\gamma_i = [0,a_i, a_{i+1}, \ldots]$.
Second,
if the $\gamma$-Ostrowski representation of $m$
is $\sum_{0 \leq i \leq t} e_i q_i$, then
$$ \{ m \gamma \} = \sum_{0 \leq i \leq t} e_i (q_i \gamma - p_i) =
\sum_{0 \leq i \leq t} 
(-1)^i e_i \gamma_0 \gamma_1 \cdots \gamma_i.$$
(Here we are using the ``circular'' representation of fractional parts
where everything is understood to be taken mod $1$).

Now the particular $\gamma$ that concerns us has a purely periodic
continued fraction expansion of period length $P$, so
$\gamma_i = \gamma_{i+P}$ for all $i \geq 0$.  So there is a constant
$C<1$ such that $\gamma_i < C$ for all $i$.  

Define $m_j$ to be the integer with least-significant-digit first
Ostrowski representation
given by $x y^j z$.  Then from above we have
\begin{equation}
\{ m_j \gamma \} = \sum_{0 \leq i \leq t}
(-1)^i e_i \gamma_0 \gamma_1 \cdots \gamma_i 
\label{mjgamma}
\end{equation}
where $e_0 e_1 \cdots e_t = x y^j z$, and
further $\lim_{j \rightarrow \infty} \{ m_j \gamma \} = \{ -\beta \}$.

Let the infinite word $f_0 f_1 \cdots = xy^\omega = xyyy\cdots$.
Consider the sum
\begin{equation}
S = \sum_{i \geq 0}
(-1)^i f_i \gamma_0 \gamma_1 \cdots \gamma_i .
\label{seq}
\end{equation}
By comparison with Eq.~\eqref{mjgamma} we see that
$|S - \{ m_j \gamma \}| = O(C^{|xy^j|})$, which is exponentially
small in $j$.  It follows that
$S = \lim_{j \rightarrow \infty} \{ m_j \gamma \} = \{ -\beta \}$.

Now the terms $f_i$ are ultimately periodic with period $|y|$,
and the $\gamma_i$ are periodic with period $P$.  So 
Eq.~\eqref{seq} expresses $S$ as the sum of finitely many 
geometric series of the form $\sum_{i \geq 0}
c_{\ell} \gamma_\ell^{P|y|i}$. Here each $c_{\ell}$ is
a finite product of the $\gamma_t$, and hence in $\Que(\gamma)$.
By the elementary formula for the sum of a geometric series,
each series sum is in $\Que(\gamma)$, and hence so is $S = \{ -\beta \}$ and
hence so is $\beta$ itself.  It is now easy to see that
the specific expression of
$\beta$ in terms of $\gamma$ is computable from the automaton.
\end{proof}

\section*{Acknowledgments}

We are grateful to Aaron Barnoff for some corrections and for a question that led to the formulation and proof of Theorem~\ref{main}, and to Jean-Paul Allouche for suggesting we look at the papers \cite{Hildebrand&Li&Li&Xie:2019,Russo&Schwiebert:2011,Gunaydin&Ozsahakyan:2022}.

We acknowledge the support of the 
Natural Sciences and Engineering Research Council of Canada (NSERC).
Nous remercions le Conseil de recherches en sciences naturelles et en g\'enie du Canada (CRSNG) de son soutien.
LS was supported by NSERC grant
RGPIN-2025-04875 and
JOS was supported by grants RGPIN-2018-04118 and
RGPIN-2024-03725.

\end{document}